\documentclass[11pt,reqno]{amsart}
\usepackage{diagrams}
\usepackage{array}
\usepackage{amssymb, mathtools}
\usepackage[pagewise]{lineno}
\numberwithin{equation}{section}
\usepackage{cite}
\usepackage{url}
\usepackage{graphicx}
\usepackage[all]{xy}
\usepackage{listings}
\makeatletter
\renewcommand*\env@matrix[1][*\c@MaxMatrixCols c]{%
	\hskip -\arraycolsep
	\let\@ifnextchar\new@ifnextchar
	\array{#1}}
\makeatother

\theoremstyle{definition}
\newtheorem{thm}{Theorem}[section]

\newtheorem{exa}[thm]{Example}
\newtheorem{prop}[thm]{Proposition}

\newtheorem{rem}[thm]{Remark}

\DeclareMathOperator{\mo}{\mathcal{O}}

\newcommand{\mr}[1]{\mathrm{#1}}
\newcommand{\mb}[1]{\mathbb{#1}}

\newcommand{\mc}[1]{\mathcal{#1}}

\begin{document}
	
	\title[Lefschetz theorem]{Failure of Lefschetz hyperplane theorem}
	
	\author{Ananyo Dan}

	\address{School of Mathematics and Statistics, University of Sheffield, Hicks building, Hounsfield Road, S3 7RH, UK}
	
	\email{a.dan@sheffield.ac.uk}
	
	\thanks{}
	
	
	
	\date{\today}
	
	\begin{abstract}
	 In this article, we give a counterexample to the Lefschetz hyperplane theorem for
	 non-singular quasi-projective varieties. 
	 A classical result of Hamm-L\^{e} shows that Lefschetz hyperplane theorem can hold for hyperplanes in general position.
	 We observe that the condition of ``hyperplane'' is strict in the sense that it is not possible to replace it 
	 by higher degree hypersurfaces.
	 The counterexample is very simple: projective space minus finitely many points. Moreover, as an intermediate step 
	 we prove that the Grothendieck-Lefschetz theorem also fails in the quasi-projective case.
	\end{abstract}

	\maketitle
	
	\section{Introduction}
	The underlying field will always be $\mb{C}$. 
	Consider a non-singular, projective variety $Y$ of dimension $n$.
  The \emph{Lefschetz hyperplane theorem} (LHT) states that for any hypersurface $X \subset Y$ with $\mo_X(Y)$
  very ample, the restriction morphism
  \begin{equation}\label{eq:lef}
   H^k(Y,\mb{Z}) \to H^k(X, \mb{Z})\, \mbox{ is an isomorphism for all } k<n-1\, \mbox{ and injective for }k=n-1.
  \end{equation}
    If $Y$ is the projective space, then the theorem extends further. In particular, the restriction from 
    $H^{n-1}(\mb{P}^n)$ to $H^{n-1}(X)$ is an isomorphism for a very general hypersurface $X$.
     The geometry of the locus of hypersurfaces where this isomorphism fails (also known as the Noether-Lefschetz locus),
     has been extensively studied \cite{voilieu, v2, ca1, M5, D3, Dcont}.
     It is therefore evident that the failure of the Lefschetz hyperplane theorem  can give rise to important questions in 
     Hodge theory and deformation theory. The goal of this article is to investigate the failure of this theorem in the 
     quasi-projective case.
     
     It was observed by Hamm and L\^{e} \cite{hamm, hamm2} that if a hyperplane section $X$
     in a quasi-projective variety $Y$ is in ``general'' position, then \eqref{eq:lef} holds true. 
     The criterion for general position, is given explicitly in terms of a Whitney stratification of $Y$ (see \S \ref{sec:hl}). 
     This leads to the natural question:
     
     {\bf{Question:}} Is the Hamm-L\^{e} theorem (Theorem \ref{thm:hamm}) true if we replace ``hyperplane'' by 
	\emph{higher degree hypersurface}?
   
   This is true in the case when $Y$ is a projective, non-singular variety.
   Surprisingly, this can fail even if $Y$ is the complement of a single point in a projective space.
   In particular, we give an example of a higher degree hypersurface which satisfies all the conditions in 
   the Hamm-L\^{e} theorem  except for being a hyperplane. Yet, in this case LHT fails.
   We now discuss this in details. 
  Recall, a projective variety $X$ is called \emph{non-factorial} if the 
  rank of the divisor class group $\mr{Div}(X)$ (i.e., the free abelian group of divisors on $X$ 
  modulo linear equivalence) is not the same as the rank of the Picard group $\mr{Pic}(X)$. We prove:
  
  \begin{thm}\label{thm1}
	Let $X \subset \mb{P}^n$ be a non-factorial hypersurface with isolated singularities  with $n \ge 4$. Denote by $X_{\mr{sing}}$ the 
	singular locus of $X$. Then, the natural restriction morphism 
	\[H^2(\mb{P}^n \backslash X_{\mr{sing}}, \mb{Z}) \to H^2(X \backslash X_{\mr{sing}}, \mb{Z})\]
	is not surjective.
	\end{thm}

	Using this theorem we now give an explicit example.
	
	\begin{exa}\label{exa01}
	Let $X \subset \mb{P}^4$ be a hypersurface defined by the equation $X_0^2+X_1^2+X_2^2+X_3^2$, 
	where $X_0,...,X_4$ are the coordinates on $\mb{P}^4$. Clearly, $X$ has exactly one singular point
	$x = [0:0:0:0:1]$. The divisor class group $\mr{Div}(X)$ is isomorphic to $\mb{Z} \oplus \mb{Z}$
	(see \cite[Ex. II.$6.5$]{R1}). By Lefschetz hyperplane theorem, we have $H^2(X,\mb{Z}) \cong \mb{Z}$.
	Using the exponential exact sequence, one can check that $\mr{Pic}(X) \cong \mb{Z}$. Hence, $X$ is 
	non-factorial. Theorem \ref{thm1} then implies that the restriction morphism from 
	$H^2(\mb{P}^4 \backslash \{x\}, \mb{Z})$ to $H^2(X \backslash \{x\}, \mb{Z})$ is not surjective.	
	\end{exa}
	
	As an intermediate step we show that the Grothendieck-Lefschetz theorem \cite{sga2}
	fails in the quasi-projective case (see Remark \ref{rem1}).

	\emph{Acknowledgement}:
	I thank Dr. I. Kaur for discussions.	The author was funded by EPSRC grant number EP/T019379/1. 
	
	\section{On the Hamm-L\^{e} result}
	In \cite{hamm}, Hamm and L\^{e}
	proved a version of the Lefschetz hyperplane theorem for quasi-projective varieties (see Theorem \ref{thm:hamm}
	below). The proof follows in two stages. We use notations as in \S \ref{sec:se} below.
	The first step is to check that for all $i \le \dim(Y)-2$, 
	$H^i(Y\backslash Z)$ (resp. $H^{m-1}(Y\backslash Z)$) is isomorphic to (resp. contained in) 
	the $i$-th (resp. $(m-1)$-th) cohomology of $V_r(L) \cap (Y \backslash Z)$,
	for some neighbourhood $V_r(L)$ of $L$ of ``radius'' $r$, for almost all $r>0$ (see \cite[Theorem $1.1.1$]{hamm}).
	The second step is to check whether $L \cap (Y \backslash Z)$ is a deformation retract of
	$V_r(L) \cap (Y \backslash Z)$. One observes that this holds true if $L$ is in a ``general'' position.
	An explicit description of the general position will be mentioned in Theorem \ref{thm:hamm} below.

     \subsection{Setup}\label{sec:se}
     Let $Y$ be a projective subvariety of dimension $m$ in $\mb{P}^n$, $Z \subset Y$ be an algebraic subspace and 
     $L \subset \mb{P}^n$ a hyperplane in $\mb{P}^n$ such that $Y \backslash (Z \cup L)$ is non-singular.
     Consider a stratification $\{Y_i\}_{i \in I}$ of $Y$ satisfying the following conditions:
     
     \begin{enumerate}
      \item each $Y_i$ is a real semi-algebraic subset of $Y$,
      \item $\{Y_i\}$ is a Whitney stratification,
      \item $Z$ is a union of some of the strata,
      \item the stratification satisfies the Thom condition for the following function:
      \[\tau:Y \to \mb{R},\, \mbox{ sending } y \in Y\, \mbox{ to } \frac{\sum\limits_{i=1}^k |f_i(y)|^{2d/d_i}}{\sum\limits_{i=0}^n |y_i|^{2d}},\, \mbox{ where } y=(y_1,...,y_n),\]
      $Z$ is defined by the homogeneous polynomials $f_1,...,f_k$ of degrees $d_i$, respectively and $d$ is the l.c.m. 
      of the       $d_i$'s. See \cite[\S $1.4.4$]{le1} for the precise definition. 
           \end{enumerate}

     \subsection{On the Hamm-L\^{e} result}\label{sec:hl}
     
     Let $\Omega$ be the set of complex projective hyperplanes of $\mb{P}^n$ transverse to all the strata $Y_i$.
     
     \begin{thm}{(Hamm-L\^{e} \cite[Theorem $1.1.3$]{hamm})}\label{thm:hamm}
      Assume that $Y \backslash Z$ is non-singular. Then, for any $L \in \Omega$ we have 
      \[H^k(Y \backslash Z, L \cap (Y \backslash Z))=0\, \mbox{ for all } k \le m-1.\]
      In other words, the natural morphism from $H^k(Y \backslash Z, \mb{Z})$ to $H^k(L \cap (Y \backslash Z), \mb{Z})$ 
      is an isomorphism
      for all $k \le m-2$ and injective for $k=m-1$.
     \end{thm}

	We now write the stratification relevant to Example \ref{exa01}.
	
	\begin{rem}
	 Take $Y=\mb{P}^4 \subset \mb{P}^5$ defined by $z_5=0$, where $z_i$ are 
	 the coordinates on $\mb{P}^5$. Take $Z:=[0,0,0,0,1,0]$ the closed point in $Y$. 
	 Take the stratification of $Y$ consisting of 
	 \[(Y \backslash Z) \coprod Z.\]
	 Then, the equations defining $Z$ in $\mb{P}^5$ are given by $f_i:= z_i$ for $0 \le i \le 3$ and $f_5:= z_5$.
	 The function $\tau$ is simply
	 \[\tau:= \frac{|z_5|^2+\sum\limits_{i=0}^3 |z_i|^2}{\sum\limits_{i=0}^5 |z_i|^2}.\]
	 Note that this stratification satisfies conditions $(1)$-$(4)$ in \S \ref{sec:se} above, with the stratification on 
	 $\mb{R}$ given by $\mb{R} \backslash \{0\} \coprod \{0\}$. Finally, note that the hypersurface $X$ in $\mb{P}^5$ defined 
	 by $z_0^2+z_1^2+z_2^2+z_3^2+z_5^2$ is singular at the point $Z$. As a result $X$ is transverse to all the strata of $Y$.
	 We will observe in Theorem \ref{thm1} that if we replace $L$ in Theorem \ref{thm:hamm} above by $X$, then the conclusion fails.
	\end{rem}

	\section{Proof of Main theorem}
	We will assume that the reader has basic familiarity with local cohomology.
	See \cite{grh} for basic definitions and results in this topic.
	
	Let $X \subset \mb{P}^n$ be a non-factorial hypersurface with isolated singularities  with $n \ge 4$. Denote by $X_{\mr{sing}}$ the 
	singular locus of $X$, $Y:= \mb{P}^n  \backslash X_{\mr{sing}}$ and $X_{\mr{sm}}:=
	X  \backslash X_{\mr{sing}}$. We first show:
	
	\begin{prop}\label{prop:van}
	 The cohomology groups $H^1(\mo_Y), H^2(\mo_Y)$ and $H^1(\mo_{X_{\mr{sm}}})$ all vanish, in both
	analytic  as well as Zariski topology.
	\end{prop}

	\begin{proof}
	 Recall, the long exact sequence for local cohomology groups, which exists in both topologies (see \cite[Corollary $1.9$]{grh}):
	 \[ ... \to H^1(\mo_{\mb{P}^n}) \to H^1(\mo_Y) \to H^2_{X_{\mr{sing}}}(\mo_{\mb{P}^n}) \to 
	 H^2(\mo_{\mb{P}^n}) \to H^2(\mo_Y) \to H^3_{X_{\mr{sing}}}(\mo_{\mb{P}^n}) \to ... 	 \]
Recall, $H^1(\mo_{\mb{P}^n})=0=H^2(\mo_{\mb{P}^n})$. By Serre's GAGA, $H^1(\mo_{\mb{P}^n}^{^{{\mr{an}}}})=0=H^2(\mo_{\mb{P}^n}^{^{\mr{an}}})$.
To prove the vanishing of 
 $H^1(\mo_Y)$ and $H^2(\mo_Y)$, we simply need to prove the vanishing of 
 $H^i_{X_{\mr{sing}}}(\mo_{\mb{P}^n})$ for $i=2,3$ in both topologies.
 
 Consider the spectral sequence (see \cite[Proposition $1.4$]{grh}):
 \begin{equation}\label{eq:spe}
  E_2^{p,q}=H^p(\mb{P}^n, \mc{H}^q_{X_{\mr{sing}}}(\mo_{\mb{P}^n})) \Rightarrow H^{p+q}_{X_{\mr{sing}}}(\mo_{\mb{P}^n}).
 \end{equation}
 We are interested in the cases when $p+q$ equals $2$ or $3$. Since $n \ge 4$ and $X_{\mr{sing}}$ are closed points, we have (see \cite[Proposition $1.2$]{yoshi})
 \[\mc{H}^q_{X_{\mr{sing}}}(\mo_{\mb{P}^n})=0 \mbox{ for } q \le 3.\]
 This implies that $E_2^{p,q}=0$ for $p+q$ equals $2$ or $3$. Hence the spectral sequence degenerates at $E_2$ in this case and 
 $H^i_{X_{\mr{sing}}}(\mo_{\mb{P}^n})=0$ in both topologies. This proves the vanishing of $H^1(\mo_Y)$ and $H^2(\mo_Y)$.
 
 The proof for the vanishing of $H^1(\mo_{X_{\mr{sm}}})$ follows similarly. In particular, using \cite[Corollary $1.9$]{grh},
 it suffices to check the vanishing of $H^1(\mo_X)$ and $H^2_{X_{\mr{sing}}}(\mo_X)$. Since $X$ is a hypersurface in 
 $\mb{P}^n$ and $n \ge 4$, $H^1(\mo_X)=0$. By Serre's GAGA, $H^1(\mo_X^{^{\mr{an}}})=0$. To prove the vanishing of 
 $H^2_{\mr{sing}}(\mo_X)$ use the spectral sequence \eqref{eq:spe} above after replacing $\mb{P}^n$ by $X$ and $p+q=2$. 
 Since $\dim X \ge 3$, \cite[Proposition $1.2$]{yoshi} implies that $\mc{H}^q_{X_{\mr{sing}}}(\mo_X)=0$ for $q \le 2$.
 This implies that the spectral sequence degenerates at $E_2$ and $H^2_{X_{\mr{sing}}}(\mo_X)=0$ in both topologies. 
 Hence, $H^1(\mo_{X_{\mr{sm}}})=0$ in both topologies. This proves the proposition.
	\end{proof}

	\begin{proof}[Proof of the main theorem]
	We prove the theorem by contradiction. Suppose that the restriction morphism from $H^2(Y,\mb{Z})$ to $H^2(X, \mb{Z})$ is surjective.
	Comparing the long exact sequences associated to the exponential exact sequence for $Y$ and $X_{\mr{sm}}$ we get the following 
	diagram where the horizontal rows are exact:
	\begin{equation}\label{eq:diag}
	 \begin{diagram}
	   H^1(\mo_Y)&\rTo&H^1(\mo_Y^*) &\rTo^{\partial_1}&H^2(Y,\mb{Z})&\rTo&H^2(\mo_Y)\\
	   \dTo&\circlearrowleft& \dTo^{\rho'}&\circlearrowleft& \dTo^{\rho}&\circlearrowleft&\dTo\\
	   H^1(\mo_{X_{\mr{sm}}})&\rTo&H^1(\mo_{X_{\mr{sm}}}^*) &\rTo^{\partial_2}&H^2(X_{\mr{sm}},\mb{Z})&\rTo&H^2(\mo_{X_{\mr{sm}}})	   
	  \end{diagram}
	\end{equation}
Using the vanishing results from Proposition \ref{prop:van}, we conclude that $\partial_1$ is an isomorphism and $\partial_2$ is injective.
By assumption, $\rho$ is surjective. We claim that $\rho'$ is surjective. Indeed, given $\alpha \in H^1(\mo_{X_{\mr{sm}}}^*)$, 
the surjectivity of $\rho$ implies that there exists $\beta \in H^2(Y,\mb{Z})$ such that $\rho(\beta)=\partial_2(\alpha)$. 
Since $\partial_1$ is an isomorphism, there exist $\alpha' \in H^1(\mo_Y^*)$ mapping to $\beta$ via $\partial_1$. 
Using the injectivity of $\partial_2$ and the commutativity of the middle square, we have $\rho'(\alpha')=\alpha$. This proves the claim.

	 Since $\rho'$ is surjective, we have the following surjective  morphism:
	 \begin{equation}\label{eq:surj}
	  \mb{Z}=\mr{Pic}(\mb{P}^n) \cong \mr{Pic}(Y) \stackrel{\rho'}{\twoheadrightarrow} \mr{Pic}(X_{\mr{sm}}) \cong \mr{Div}(X)
	 \end{equation}
 where the second and the last isomorphisms follow from the fact that $X_{\mr{sing}}$ is of codimensional at least $2$ in $X$ and $\mb{P}^n$. 
 By Lefschetz hyperplane theorem, we have $H^2(X,\mb{Z}) \cong H^2(\mb{P}^n, \mb{Z}) = \mb{Z}$, generated by the class of the hyperplane section.
  Note that, $H^1(\mo_X)$ and $H^2(\mo_X)$ vanish (use \cite[Ex. III.$5.5$]{R1}
 and $n \ge 4$). Using the exponential short exact sequence for $X$, we conclude that $\mr{Pic}(X) \cong \mb{Z}$. 
 Combining with \eqref{eq:surj}, this implies $\mr{rk}\, \mr{Div}(X) = \mr{rk}\, \mr{Pic}(X)$. But this contradicts the fact that $X$ is non-factorial. 
 Hence, the restriction morphism from $H^2(Y,\mb{Z})$ to $H^2(X, \mb{Z})$ cannot be surjective. This proves the theorem.
	\end{proof}
	
	\begin{rem}\label{rem1}
	 Let $X$ be as in Theorem \ref{thm1}. Then, the restriction morphism 
	 \[\mr{Pic}(\mb{P}^n \backslash X_{\mr{sing}}) \to \mr{Pic}(X \backslash X_{\mr{sing}})\]
	 is not surjective. Indeed, 
	 \[\mr{Pic}(\mb{P}^n \backslash X_{\mr{sing}}) \cong \mr{Pic}(\mb{P}^n) \cong \mb{Z} \mbox{ and }
	  \mr{Pic}(X \backslash X_{\mr{sing}}) \cong \mr{Div}(X).\]
      By Lefschetz hyperplane theorem for projective hypersurfaces, we have $\mr{Pic}(X) \cong \mb{Z}$.
      Since $X$ is non-factorial, the rank of $\mr{Div}(X)$ is not the same as that of $\mr{Pic}(X)$.
      Therefore, $\mr{Pic}(\mb{P}^n \backslash X_{\mr{sing}})$ cannot be isomorphic to 
      $\mr{Pic}(X \backslash X_{\mr{sing}})$.
	\end{rem}

\end{document}